\documentclass[a4paper, reqno]{amsart}
\usepackage{amssymb, amsmath, amsthm, enumerate, mathrsfs, mathscinet, mathtools, pgfplots, standalone, tikz}

\numberwithin{equation}{section}
\newtheorem{theorem}{Theorem}[section]
\newtheorem{corollary}[theorem]{Corollary}
\newtheorem{lemma}[theorem]{Lemma}
\newtheorem{proposition}[theorem]{Proposition}

\theoremstyle{definition}
\newtheorem{definition}[theorem]{Definition}
\newtheorem{remark}[theorem]{Remark}

\newtheorem*{acknowledgment}{Acknowledgment}

\newcommand{\C}{\mathbb{C}}

\newcommand{\Z}{\mathbb{Z}}
\newcommand{\N}{\mathbb{N}}

\newcommand{\ud}{\mathrm{d}}

\pgfplotsset{compat=1.6}

\pgfplotsset{soldot/.style={color=blue,only marks,mark=*}} \pgfplotsset{rsoldot/.style={color=red,only marks,mark=*}} 

\begin{document}

\title[Factorisations of monomials]{Counting factorisations of monomials over rings of integers modulo $N$}

\author[J. Hickman]{ Jonathan Hickman }
\address{Jonathan Hickman: Eckhart Hall Room 414, Department of Mathematics, University of Chicago, 5734 S. University Avenue, Chicago, Illinois,  60637, US.}
\email{jehickman@uchicago.edu}

\author[J. Wright]{ James Wright }
\address{James Wright: Room 4621, James Clerk Maxwell Building, The King's Buildings, Peter Guthrie Tait Road, Edinburgh, EH9 3FD, UK.}
\email{j.r.wright@ed.ac.uk}

\begin{abstract} A sharp bound is obtained for the number of ways to express the monomial $X^n$ as a product of linear factors over $\Z/p^{\alpha}\Z$. The proof relies on an induction-on-scale procedure which is used to estimate the number of solutions to a certain system of polynomial congruences. The method also applies to more general systems of polynomial congruences that satisfy a \emph{non-degeneracy} hypothesis. 
\end{abstract}

\maketitle




\section{Introduction}

Let $\alpha, n \in \N$ and $p \in \N$ be prime. One purpose of this note is to provide a precise count for the number of factorisations of the monomial $X^n$ into linear factors over $\Z/p^{\alpha}\Z$. That is, (after normalising) one wishes to determine the value of
\begin{equation*}
\mathbf{N}(\vec{0}_n; p^{\alpha}) := p^{-\alpha n} \#\Big\{(x_1, \dots, x_n) \in [\Z/p^{\alpha}\Z]^n : X^n \equiv \prod_{j=1}^n (X-x_j) \bmod p^{\alpha} \Big\}.
\end{equation*}
If $\alpha = 1$, then $\mathbb{F}_p := \Z/p\Z$ is a field and the polynomial ring $\mathbb{F}_p[X]$ is a unique factorisation domain and so there is only one possible factorisation of $X^n$. In general, however, there are many different factorisations: for example, 
\begin{equation*}
X^2 \equiv (X-3)(X-6) \bmod 9. 
\end{equation*}

Unfortunately, the general statement of the results is slightly involved. Things become much cleaner, however, if $n$ is assumed to be a triangular number, and it is instructive to first consider this case. In particular, letting $\triangle_{r} := \frac{r(r+1)}{2}$ denote the $r$th triangular number, the following estimate holds. 

\begin{proposition}\label{example proposition} If $n = \triangle_r$ for some $r \in \N$ with $r \geq 2$ and $p$ is a sufficiently large prime, then\footnote{Given a (possibly empty) list of objects $L$, for real numbers $A, B \geq 0$ the notation $A \ll_L B$ or $B \gg_L A$ signifies that $A \leq C_LB$ for some constant $C_L\geq 0$ depending only on the objects in the list. Furthermore, $A \sim_L B$ signifies that $A \ll_L B \ll_L A$.}   
\begin{equation*}
\mathbf{N}(\vec{0}_n; p^{\alpha}) \ll_n  \alpha p^{-\alpha r}
\end{equation*}
holds for all $\alpha \in \N$. 
The result is sharp for all $n \neq 3$ in the sense that the reverse inequality also holds for infinitely many $\alpha$. 
\end{proposition} 

When $n$ is not a triangular number, the asymptotics of $\mathbf{N}(\vec{0}_n; p^{\alpha})$ are not readily expressed in a single compact formula. In fact, for $n=2$ it is a simple matter to see that $\mathbf{N}(\vec{0}_2; p^{\alpha}) = p^{-3\alpha/2}$ when $\alpha$ is even and $\mathbf{N}(\vec{0}_2; p^{\alpha}) = p^{-(3\alpha + 1)/2}$ when $\alpha$ is odd.

In order to state the general form of Proposition \ref{example proposition}, first define
\begin{equation*}
e_n(\alpha, r) := r\alpha + (n - \triangle_{r}) \cdot  \left\{\begin{array}{ll}
\lceil \frac{\alpha}{r+1} \rceil &\textrm{if $\triangle_{r} \leq n$} \\[6pt]
\lceil \frac{\alpha}{r}\rceil  - 1  &\textrm{if $\triangle_{r} \geq n$}
\end{array}\right. 
\end{equation*}
for $0 \leq r \leq n-1$. Let $e_n(\alpha) := \min \{ e_n(\alpha, r) : r \in R_n(\alpha)\}$ where
\begin{equation*}
R_n(\alpha) := \{0\} \cup \big\{ 1 \leq r \leq n-1 : \lceil \tfrac{\alpha}{r+1}\rceil < \lceil \tfrac{\alpha}{r}\rceil \big\}.
\end{equation*}
and let
\begin{equation*}
[\alpha]^{\delta_n(\triangle)} := \left\{ \begin{array}{ll}
\alpha & \textrm{if $n = \triangle_r$ for some $r \geq 1$ and $r \in R_n(\alpha)$} \\
1 & \textrm{otherwise}
\end{array}\right. . 
\end{equation*}
The main theorem is as follows.  

\begin{theorem}\label{main theorem} If $n \in \N$, $n \neq 3$ and $p$ is a sufficiently large prime, then
\begin{equation}\label{main theorem estimate}
\mathbf{N}(\vec{0}_n; p^{\alpha}) \sim_{n}  [\alpha]^{\delta_n(\triangle)} p^{-e_n(\alpha)}
\end{equation}
holds for all $\alpha \in \N$. 

\end{theorem}

Explicitly, the proof shows that the upper bound in \eqref{main theorem estimate} holds if $p > n$. Curiously, the $n=3$ case behaves differently and, in particular, the asymptotics for $\mathbf{N}(\vec{0}_3; p^{\alpha})$ depend on the congruence class of $p$ modulo 3: see Lemma \ref{cubic lemma}, below. 

The definition of the exponent $e_n(\alpha)$ is somewhat complicated and it is useful to consider some examples. For instance, when $n=2$ it follows that $e_2(\alpha,0) = 2\alpha$ and $e_2(\alpha,1) = \alpha + \lceil \frac{\alpha}{2} \rceil$ for all $\alpha \in \N$, whilst $R_2(1) = \{0\}$ and $R_2(\alpha) = \{0,1\}$ when $\alpha\geq 2$. Thus, one deduces that $e_2(\alpha) = e_2(\alpha,1)$ and $[\alpha]^{\delta_2(\triangle)} = 1$ for all $\alpha \in \N$. Therefore, $[\alpha]^{\delta_2(\triangle)} p^{-e_2(\alpha)} = p^{-3\alpha/2}$ when $\alpha$ is even and $[\alpha]^{\delta_2(\triangle)} p^{-e_2(\alpha)} = p^{-(3\alpha + 1)/2}$ when $\alpha$ is odd and so \eqref{main theorem estimate} yields
\begin{equation*}
\mathbf{N}(\vec{0}_2; p^{\alpha}) \ \sim \ p^{-3\alpha/2} \ \ \textrm{or} \ \ p^{-(3\alpha + 1)/2},
\end{equation*}
depending on whether $\alpha$ is even or odd, respectively. As noted earlier, these asymptotics are in fact an equality.

It is also instructive to understand how Theorem \ref{main theorem} relates to Proposition \ref{example proposition}. First note that, by a simple computation, 
\begin{align*}
e_n(\alpha, r-1) - e_n(\alpha, r) &\geq 0 \quad \textrm{for $\triangle_{r} \leq n$},\\
e_n(\alpha, r+1) - e_n(\alpha, r) &\geq 1 \quad \textrm{for $\triangle_r \geq n$}   
\end{align*} 
and, consequently, $e_n(\alpha) \geq \min \{ e_n(\alpha, r_n^-), e_n(\alpha, r_n^+)\}$ with equality if $r_n^-, r_n^+ \in R_n(\alpha)$ where 
\begin{equation*}
r_n^- := \max \{ r \geq 0 : \triangle_r \leq n\} \quad \textrm{and} \quad r_n^+ := \min \{ r \geq 0 : \triangle_r \geq n\}.
\end{equation*}
In general, either $e_n(\alpha, r_n^-)$ or $e_n(\alpha, r_n^+)$ can achieve the minimum: compare, for instance, the case $n=5$, $\alpha = 3k$ with $n=5$, $\alpha = 3k+1$ for any $k \in \N$. However, if $n = \triangle_r$ is triangular, then $r_n^- = r_n^+ = r$ and so  
\begin{equation}\label{triangular exponent}
e_n(\alpha) \geq e_n(\alpha, r) = r\alpha \qquad \textrm{with equality if $r \in R_n(\alpha)$} .
\end{equation}
Note that the right-hand exponent is precisely that appearing in Proposition \ref{example proposition}. It is easy to verify that 
\begin{align*}
\big\{ r \in \N : \lceil\tfrac{\alpha}{r+1}\rceil \not< \lceil \tfrac{\alpha}{r}\rceil \} &= \bigcup_{k\in \N} \big\{ r \in \N : \lceil\tfrac{\alpha}{r+1}\rceil = \lceil \tfrac{\alpha}{r}\rceil = k \big\} \\
&= \bigcup_{k\in \N} \big\{ r \in \N : \alpha \leq k r \leq \alpha + r - k \big\}
\end{align*} 
and so if $\alpha$ is sufficiently large (in particular, if $\alpha > (n-1)^2$), then $R_n(\alpha) = \{0,1,\dots, n-1\}$ and one has $e_n(\alpha) = \min \{ e_n(\alpha, r_n^-), e_n(\alpha, r_n^+)\}$ with $e_n(\alpha) = \alpha r$ if $n = \triangle_r$. On the other hand, if $\alpha = 1$, then $R_n(1) = \{0\}$ for all $n \in \N$ and so $e_n(1) = e_n(1, 0) = n$, which is consistent with the unique factorisation property.  

Observe that Theorem \ref{main theorem} can be recast as an estimation of the number of solutions to a certain system of congruence equations. In particular, by the classical Newton--Girard formul\ae\, (see, for instance, \cite[(2.11$^\prime$)]{Macdonald1995}), if $p > n$, then $\mathbf{N}(\vec{0}_n; p^{\alpha})$ is precisely the normalised number of solutions in $[\Z/p^{\alpha}\Z]^n$ to the system
\begin{equation}\label{original system}
P_k(X_1, \dots, X_n) \equiv 0 \mod p^{\alpha} \qquad \textrm{for $1 \leq k \leq n$} 
\end{equation}
where $P_k$ is the $k$th power sum, given by
\begin{equation*}
P_k(X_1, \dots, X_n) := X_1^k + \dots + X_{n}^k.
\end{equation*}

At this point some contextual remarks are in order.
 
\begin{remark}\label{main theorem remark}
\begin{enumerate}[1)]
\item\label{general polynomials remark} Rather than restrict to monomials, given an $n$-tuple
$\vec{y} = (y_1, \dots, y_n) \in [\Z/p^{\alpha}\Z]^n$, one could estimate the normalised count $\mathbf{N}(\vec{y};p^{\alpha})$ of the number of factorisations 
\begin{equation*}
\prod_{j=1}^n (X-y_j) \equiv \prod_{j=1}^n (X-x_j) \mod p^{\alpha}.
\end{equation*}
For $p > n$, this is equivalent to counting solutions to the system 
\begin{equation*}
P_k(X_1, \dots, X_n) \equiv P_k(y_1, \dots, y_n) \mod p^{\alpha} \qquad \textrm{for $1 \leq k \leq n$.} 
\end{equation*}
If the components of $\vec{y}$ are well-separated in the $p$-adic sense, then $\mathbf{N}(\vec{y}; p^{\alpha})$ can be
bounded (in fact, explicitly determined) via Hensel lifting, as observed in \cite{Hickman} (see also \cite{Wooley1996}, which treats very general systems of congruences under much stronger `non-degeneracy' hypotheses). Theorem \ref{main theorem} corresponds to the case $\vec{y} = \vec{0}_n$, which is a highly degenerate situation where there is no $p$-adic separation between the components of $\vec{y}$, and therefore acts as a counterpoint to the observations of \cite{Hickman}.  
\item The problem of counting factorisations of polynomials arose naturally in a recent study of the so-called \emph{Fourier restriction phenomenon} for curves over $\Z/N\Z$ \cite{Hickman} (see also \cite{Hickman3}). Fixing a polynomial curve $\gamma \colon   \Z/N\Z \to [ \Z/N\Z]^n$, the Fourier restriction problem involves the estimation of weighted exponential sums
\begin{equation}\label{extension operator}
\mathcal{E}G(\vec{y}\,) := \sum_{x \in \Z/N\Z} G(x)e^{2 \pi i \gamma(x) \cdot \vec{y}/N}
\end{equation}
defined for any coefficient function $G \colon \Z/N\Z \to \C$. In \cite{Hickman} a conjectural upper bound for $\mathbf{N}(\vec{y}; p^{\alpha})$ is stated and a proof is given under additional hypotheses on $\vec{y}$ (see item 1). Moreover, good control over $\mathbf{N}(\vec{y};p^{\alpha})$ is shown to imply favourable estimates for \eqref{extension operator} in the prototypical case where $\gamma(x) := (x, x^2, \dots, x^n)$. It is remarked that for $\vec{y} = \vec{0}_n$ the conjectured upper bound for $\mathbf{N}(\vec{0}_n; p^{\alpha})$ from \cite{Hickman} is trivial. The strengthened estimate of Theorem \ref{main theorem} does not appear to be directly applicable to Fourier restriction theory, but it is likely that the methods of proof will be useful in future studies. Furthermore, the problem of counting factorisations of monomials over $\Z/N\Z$ is arguably of some inherent interest. 
\item Generalising the above notation, let $\mathbf{N}(\vec{0}_n; N)$ denote the normalised number of factorisations of $X^n$ modulo $N$ for $N \in \N$. By the Chinese remainder theorem, $\mathbf{N}(\vec{0}_n; N)$ is a multiplicative function of $N$. If all the prime factors $p$ of $N$ satisfy $p > n$, then Theorem \ref{main theorem} implies that for all $\varepsilon > 0$ there exists a constant $C_{\varepsilon, n}$ such that
\begin{equation}\label{general N estimate}
\mathbf{N}(\vec{0}_n; N) \leq C_{\varepsilon, n} N^{\varepsilon} \prod_{p \mid N} p^{-e_n(\mathrm{ord}_p(N))} 
\end{equation}
where the product is over all prime factors of $N$ and the integer $\mathrm{ord}_p(N)$ is the multiplicity of the prime divisor $p$ (so that $N = \prod_{p \mid N} p^{\mathrm{ord}_p(N)}$). Indeed, Theorem \ref{main theorem} immediately yields \eqref{general N estimate} with $C_{\varepsilon, n}N^{\varepsilon}$ replaced with 
\begin{equation*}
C_{n}^{\omega(N)} \prod_{p \mid N} \mathrm{ord}_p(N) \leq C_n^{\omega(N)} \bigg(\frac{\log N}{\omega(N)}\bigg)^{\omega(N)} 
\end{equation*}
where $\omega(N) := \sum_{p \mid N} 1$ is the number of distinct prime divisors of $N$. Note that, as a simple and well-known consequence of Tchebychev's theorem on the size of the prime counting function, 
\begin{equation*}
\omega(N) \ll \frac{\log N}{\log \log N}.
\end{equation*}
By combining these observations, and considering the cases $\omega(N) \geq \varepsilon \cdot \log N/\log \log N$ and $\omega(N) < \varepsilon \cdot \log N/\log \log N$ separately, one readily deduces \eqref{general N estimate}. 
\item The authors have not attempted to optimise the values of the implied constants in Theorem \ref{main theorem} (that is, neither the size of the dimensional constant in \eqref{main theorem estimate}, nor the lower bound on $p$). It is likely that improvements would follow from a more thorough analysis of systems of power sums over finite fields.  
\end{enumerate}
\end{remark}

Theorem \ref{main theorem} is a consequence of a more general result concerning \emph{non-degenerate} systems of congruences.\footnote{The notion of non-degeneracy discussed here is distinct from that appearing above in Remark \ref{main theorem remark} \ref{general polynomials remark}).} 

\begin{definition}\label{non-degenerate} An $m$-tuple of homogeneous polynomials $\vec{f} =(f_1, \dots, f_m) \in \Z[X_1, \dots, X_n]$ with $1 \leq m \leq n$ is said to be non-degenerate over $\mathbb{F}_p$ for a prime $p$ if for each $1 \leq r \leq m$ one has
\begin{equation*}
\mathrm{rank} \frac{\partial (f_1, \dots, f_r)}{\partial (X_1, \dots, X_n)} (x) = r
\end{equation*}
whenever $x \in \mathbb{P}^{n-1}(\overline{\mathbb{F}}_p)$ satisfies $f_k(x) = 0$ for all $1 \leq k \leq r$. 
\end{definition}

Here $\overline{\mathbb{F}}_p$ denotes the algebraic closure of the $p$-field $\mathbb{F}_p$ and $\mathbb{P}^{n-1}(\overline{\mathbb{F}}_p)$ the $(n-1)$-dimensional projective space over $\overline{\mathbb{F}}_p$. Any $f \in \Z[X_1, \dots X_n]$ can be considered a polynomial over $\mathbb{F}_p$ by reducing the coefficients modulo $p$; if $f$ is homogeneous, then it can also be considered a polynomial over $\mathbb{P}^{n-1}(\overline{\mathbb{F}}_p)$. 

For $1 \leq m \leq n$ let $\vec{f} = (f_1, \dots, f_{m}) \in \Z[X_1, \dots, X_n]^{m}$ be an $m$-tuple of polynomial mappings, where each $f_k$ is homogeneous of degree $d_k$ and $1\leq d_1 \leq \dots \leq d_m$ and suppose that $\vec{f}$ is non-degenerate 
over $\mathbb{F}_p$ for some prime $p$. For $\alpha \in \N$ one wishes to estimate 
\begin{equation*}
\mathbf{\mathcal N}(\vec{f}; p^{\alpha}) := p^{-\alpha n} \#\Big\{ \vec{x} \in [\Z/p^{\alpha}\Z]^n : f_k(\vec{x}\,) \equiv 0 \bmod p^{\alpha} \textrm{ for $1\leq k \leq m$} \Big\}.
\end{equation*}
In order to state the results, let $\sigma_{r} := \sum_{k=1}^r d_k$ for $0 \leq r \leq m$ (here $\sigma_0 :=0$) and
\begin{equation*}
\alpha_r := 
\left\{ \begin{array}{ll}
\infty & \textrm{for $r = 0$}\\
\lceil \frac{\alpha}{d_{r}}\rceil & \textrm{for $1 \leq r \leq m$} \\
0 & \textrm{for $r = m+1$}
\end{array}\right. . 
\end{equation*}
 Define
\begin{equation*}
e(\vec{f};\alpha, r) := r\alpha + (n - \sigma_{r}) \cdot  \left\{\begin{array}{ll}
\alpha_{r+1}  &\textrm{if $\sigma_{r} \leq n$} \\[6pt]
 \alpha_r  - 1  &\textrm{if $\sigma_{r} \geq n$}
\end{array}\right. 
\end{equation*}
and $e(\vec{f};\alpha) := \min \{ e(\vec{f};\alpha, r) : r \in R(\vec{f};\alpha)\}$ where
\begin{equation*}
R(\vec{f};\alpha) := \big\{0 \leq r \leq \min\{m, n-1\} : \alpha_{r+1} < \alpha_r \big\}
\end{equation*}
and
\begin{equation*}
[\alpha]^{\delta(\vec{f};\sigma)} := \left\{ \begin{array}{ll}
\alpha & \textrm{if $n = \sigma_r$ for some $r \geq 1$ and $r \in R(\vec{f};\sigma)$} \\
1 & \textrm{otherwise}
\end{array}\right. . 
\end{equation*}
The general version of Theorem \ref{main theorem} is as follows. 
\begin{theorem}\label{general theorem} With the above setup, if $p$ is a sufficiently large prime, depending on $n$ and $\deg \vec{f}$, then
\begin{equation}\label{upper}
\mathbf{\mathcal N}(\vec{f},p^{\alpha}) \ll_{n, \deg\vec{f}} [\alpha]^{\delta(\vec{f};\sigma)} p^{-e(\vec{f};\alpha)}
\end{equation}
holds for all $\alpha \in \N$. 
\end{theorem}

In many instances this theorem can also be shown to be sharp in the sense of Theorem \ref{main theorem}; see the discussion in $\S$\ref{proof section} for more details.

It is shown in $\S$\ref{application section} that the $n$-tuple of power sums $(P_1, \dots, P_n)$ is non-degenerate over $\mathbb{F}_p$ for all primes $p > n$. In this case $d_k = k$ and $\sigma_r = \triangle_r$, so that the exponent $e(\vec{f};\alpha)$ in Theorem \ref{general theorem} reduces to the exponent $e_n(\alpha)$ appearing in Theorem \ref{main theorem}. Thus, the upper bound in \eqref{main theorem estimate} is a consequence of \eqref{upper}.


The notion of non-degeneracy introduced above is very strong and it is natural to ask whether sharp bounds for $\mathbf{\mathcal N}(\vec{f}; p^{\alpha})$ can be established under weakened hypotheses. If one does not impose any kind of non-degeneracy condition, then this is a very difficult problem. Indeed, the simple case of a single homogeneous polynomial in two variables was only recently understood \cite{Wright}; the problem for a single homogeneous polynomial in $n$ variables remains open and is closely related to certain long-standing conjectures of Igusa (see \cite{Igusa1978} and also \cite{Denef1991}). Denef and Sperber \cite{Denef2001} (see also \cite{Cluckers2008, Cluckers2010}) considered the case of a single homogeneous polynomial $f$ in $n$ variables under the hypothesis that $f$ is \emph{non-degenerate with respect to its Newton diagram} (see \cite{Denef2001} for the relevant definitions). Although related, the present notion of non-degeneracy is somewhat different; for instance, $f(x,y,z) = (x-y)^2 + x z$ is non-degenerate in the sense of Definition \ref{non-degenerate} but it is not non-degenerate with respect to its Newton diagram. On the other hand, $f(x,y,z) = x y z$ is non-degenerate with respect to its Newton diagram yet
it fails to satisfy the condition in Definition \ref{non-degenerate}.


The introduction is concluded with a brief sketch of the methods used to prove Theorem \ref{general theorem}. The problem can be lifted to the $p$-adic setting and reformulated as an estimate of the Haar measure of certain sub-level sets defined over $\Z_p^n$. An induction-on-scale procedure is then applied to determine the size of these sub-level sets. The base case and inductive step for this induction-on-scale can be loosely summarised as follows:
\begin{itemize}
\item The base case corresponds to studying the system 
\begin{equation}\label{general system}
f_k(\vec{X}) \equiv 0 \bmod p \qquad \textrm{for $1 \leq k \leq m$.}
\end{equation} 
Upper and lower bounds on the number of solutions of such systems can be obtained by appealing to classical results from algebraic geometry, such as the Lang--Weil bound \cite{Lang1954}. 
\item To establish the inductive step one must verify certain transversality conditions which naturally arise in the analysis. This involves showing that certain configurations of hyperplanes in $\mathbb{F}_p^n$ are in general position. Thus, the induction-on-scale effectively reduces a non-linear problem over rings with zero divisors to a linear algebra problem over finite fields. 
\end{itemize}

This article is organised as follows: $\S$\ref{algebraic preliminaries section} and $\S$\ref{proof section} contain the proof of Theorem \ref{general theorem}. In particular, certain algebraic preliminaries are discussed in $\S$\ref{algebraic preliminaries section} whilst $\S$\ref{proof section} contains the main details of the aforementioned induction argument. In $\S$\ref{application section} Theorem \ref{general theorem} is shown to imply the upper bound in Theorem \ref{main theorem}. In $\S$\ref{lower bounds section} there is a detailed discussion of the lower bound in Theorem \ref{main theorem} and the $n=3$ case. The paper concludes with an appendix which provides details of various facts from algebraic geometry and commutative algebra used to analyse the system \eqref{general system}. 

\begin{acknowledgment} The authors would like to thank Julia Brandes and Jordan Ellenberg for interesting discussions on topics related to this project. This material is based upon work supported by the National Science Foundation under Grant No. DMS-1440140 while the first author was in residence at the Mathematical Sciences Research Institute in Berkeley, California, during the Spring 2017 semester.  
\end{acknowledgment}




\section{Algebraic Preliminaries}\label{algebraic preliminaries section}

Throughout this section let $1 \leq m \leq n$ and $\vec{f} := (f_1, \dots, f_m) \in \Z[X_1, \dots, X_n]$ be a system of homogeneous polynomials satisfying the non-degeneracy hypothesis over $\mathbb{F}_p$. The proofs of Theorem \ref{main theorem} and Theorem \ref{general theorem} will require estimates for the number of solutions to each of the partial systems of congruences 
\begin{equation}\label{partial system}
f_j(\vec{X}\,) \equiv 0 \mod p \quad \textrm{for $1 \leq j \leq r$}
\end{equation}
for $1 \leq r \leq m$.

\begin{lemma}\label{trivial solution lemma} If $r = m = n$, then the system \eqref{partial system} has a unique (trivial) solution. 
\end{lemma}

\begin{proof} If $\vec{x} \in \mathbb{F}_p^n$ satisfies $f_j(\vec{x}\,) \equiv 0 \bmod p$, then, by Euler's formula for homogeneous polynomials, $\langle \nabla f_j(\vec{x}\,), \vec{x} \rangle = 0$. Hence, the only solution to \eqref{partial system} for $r = n$ is $\vec{x} = \vec{0}$, since the non-degeneracy hypothesis implies that $\{\nabla f_j(\vec{x}\,) : 1 \leq j \leq n\}$ forms a basis of $\mathbb{F}_p^n$ whenever $\vec{x} \neq \vec{0}$.   
\end{proof}

Counting the number of solutions to \eqref{partial system} when $r < n$ is more involved and is achieved by appealing to standard estimates from algebraic geometry. For this, it will be convenient to work over projective space. In particular, for $1 \leq r \leq m$ define
\begin{equation}\label{projective variety}
V_r := \{x \in \mathbb{P}^{n-1}(\overline{\mathbb{F}}_p) : f_j(x) = 0 \textrm{ for $1 \leq j \leq r$}\}
\end{equation}
and let $V_r(\mathbb{F}_p)$ denotes the set of $\mathbb{F}_p$-rational points of $V_r$; here a point $x \in \mathbb{P}^{n-1}(\mathbb{F}_p)$ is $\mathbb{F}_p$-rational if it can be expressed in homogeneous coordinates as $x = [x_1 : \dots : x_n]$ with $x_1, \dots, x_n \in \mathbb{F}_p$.

\begin{lemma}[Schwarz--Zippel-type bound]\label{Schwarz--Zippel lemma} For $1 \leq r \leq \min\{m, n-1\}$ one has
\begin{equation*}
|V_r(\mathbb{F}_p)| \leq \big(\prod_{j=1}^r \deg f_j \big) \cdot |\mathbb{P}^{n-r-1}(\mathbb{F}_p)|.
\end{equation*}
\end{lemma}

This lemma is a direct application of a well-known \emph{Schwarz--Zippel-type bound}\footnote{The terminology comes from comparison with the classical Schwarz-Zippel bound, which essentially corresponds to the case $m=1$.} which applies to general projective varieties over $\mathbb{F}_p$: see, for instance, \cite[Corollary 2.2]{Lachaud2015}. To apply the Schwarz--Zippel bound one must demonstrate that each $V_r$ is a projective variety of dimension $n-r-1$; since $V_r$ is defined by $r$ homogeneous polynomial equations, given the non-degeneracy hypothesis it is intuitively clear that $\dim V_r = n - 1 - r$ should hold. However, the notion of dimension used here is of a precise algebraic-geometric nature and the verification of the condition $\dim V_r = n - 1 - r$ is postponed until the appendix. 

To establish the lower bound in Theorem \ref{main theorem} and the sharpness of the estimates in Theorem \ref{general theorem}, one is also required to bound $|V_r(\mathbb{F}_p)|$ from below. 

\begin{proposition}[Lang--Weil estimate]\label{classical estimates proposition} For $1 \leq r \leq  \min\{m, n-2\}$ one has
\begin{equation}\label{Lang--Weil estimate}
\big||V_r(\mathbb{F}_p)| - |\mathbb{P}^{n-r-1}(\mathbb{F}_p)|\big| \ll_{\deg\vec{f},n,r} p^{n-r-3/2}.
\end{equation}
\end{proposition}

Proposition \ref{classical estimates proposition} is a direct application of the classical estimate of Lang--Weil \cite{Lang1954}. In order to apply the Lang--Weil theorem, one must verify that the variety $V_r$ is \emph{absolutely irreducible}\footnote{See the appendix for the relevant definitions.} when $1 \leq r \leq n-2$; establishing this property requires some algebraic geometry and the proof is discussed in detail in the appendix.  It is remarked that one may obtain stronger estimates by applying the deep work of Deligne \cite{Deligne1974}; here the Lang--Weil inequality is preferred due to its relative simplicity (in particular, there exist fairly elementary proofs of the Lang--Weil theorem: see \cite{Stepanov1969, Bombieri1974}).

Finally, it is remarked that for the prototypical example of $(P_1, \dots, P_n)$ an upper bound on the number of solutions in $\mathbb{F}_p^n$ can be obtained using very elementary methods. 

\begin{lemma}\label{partial system lemma} For $1 \leq r \leq n$ one has
\begin{equation*}
\mathbf{\mathcal N}(P_1, \dots, P_r; p)\leq n! p^{-r}.
\end{equation*}
\end{lemma}

\begin{proof} Clearly one may write
\begin{equation*}
\mathbf{\mathcal N}(P_1, \dots, P_r; p) = \sum_{\vec{y} = (0, \dots, 0, y_{r+1}, \dots, y_{n}) \in \mathbb{F}_p^{n}} {\mathcal N}(\vec{P} - \vec{y}; p)
\end{equation*}
where $\vec{P} := (P_1, \dots, P_n)$. Observe that ${\mathcal N}(\vec{P} - \vec{y}; p)$ is a normalised count of $n$-tuples of roots of a fixed univariate polynomial over $\mathbb{F}_p^n$. One therefore immediately deduces that ${\mathcal N}(\vec{P} - \vec{y}; p) \leq n!p^{-n}$ and, since the above sum is over $p^{n-r}$ choices of $\vec{y}$, this concludes the proof.
\end{proof}




\section{The proof of Theorem \ref{general theorem}}\label{proof section}

The key observation in the proof of Theorem \ref{general theorem} is the following formula, which effectively reduces the problem to estimating the size of varieties over finite fields.

\begin{proposition}\label{key formula proposition} Let $n \in \N$ and $p$ be prime. If $\vec{f} = (f_1, \dots, f_{m})$ is non-degenerate over $\mathbb{F}_p$, then
\begin{equation}\label{key formula}
\mathbf{\mathcal N}(\vec{f}; p^{\alpha}) \sim_{\deg\vec{f}} \sum_{r \in R(\vec{f}; \alpha)} c_n(\alpha,r) \cdot 
\big(\mathbf{\mathcal N}(f_1, \dots, f_r;p) - p^{-n}\big) \cdot p^{-e(\vec{f};\alpha, r) +r}
\end{equation}
where $c_{n}(\alpha,r) = \alpha_r - \alpha_{r+1}$ if $\sigma_r = n$ and $1$ otherwise.
\end{proposition}

Assuming this proposition, Theorem \ref{general theorem} readily follows from the estimates discussed in the previous section.

\begin{proof}[Proof (of Theorem \ref{general theorem})] The Schwarz--Zippel-type bound (or, in the case of power sums, Lemma \ref{partial system lemma}) from the previous section imply that
\begin{equation*}
\mathbf{\mathcal N}(f_1, \dots, f_r;p) \ll_{\deg\vec{f}, n,r} p^{-r} \qquad \textrm{for $1 \leq r \leq \min\{m, n-1\}$},
\end{equation*}
whilst Lemma \ref{trivial solution lemma} shows that $\mathbf{\mathcal N}(f_1, \dots, f_n;p) = p^{-n}$. 
Applying these estimates to the formula \eqref{key formula}, one immediately deduces that
\begin{equation*}
\mathbf{\mathcal N}(\vec{f}; p^{\alpha}) \ll_{\deg\vec{f},n} [\alpha]^{\delta(\vec{f};\sigma)} p^{-e(\vec{f};\alpha)},
\end{equation*}
as required
\end{proof}
%

If $p$ is sufficiently large depending on $n$ and $\deg \vec{f}$, then Proposition \ref{key formula proposition} can be used to deduce effective lower bounds for $\mathbf{\mathcal N}(\vec{f}; p^{\alpha})$. A difficulty arises here due to the fact that the Lang--Weil estimate for $\mathbf{\mathcal N}(f_1, \dots, f_r ; p)$ is only available when $1 \leq r \leq n-2$. This causes complications if the minimum of $e(\vec{f}; \alpha, r)$ occurs when $r = n-1 \in R(\vec{f}; \alpha)$. In practice, this issue rarely manifests itself: in the prototypical case of the power sum system $\vec{f} = \vec{P}$ it only affects the $n=3$ case, which can be understood completely via a direct counting argument (see $\S$\ref{lower bounds section}). 

To make the above discussion more concrete, suppose that 
$\vec{f} = (f_1, \dots, f_{n})$ is non-degenerate over $\mathbb{F}_p$ and the degrees $d_k$ of the $f_k$ satisfy $d_1 < \dots < d_n$. If $\alpha$ is sufficiently large, depending on $\deg \vec{f}$, then $\alpha_1 > \dots > \alpha_n$ and so $R(\vec{f};\alpha) = \{0,1, \dots, n-1\}$. As in the discussion following Theorem \ref{main theorem}, one deduces that $e(\vec{f};\alpha) = \min \{ e(\vec{f};\alpha, r_n^-), e(\vec{f};\alpha, r_n^+)\}$ where 
\begin{equation*}
r_n^- := \max \{ r \geq 0 : \sigma_r \leq n\} \quad \textrm{and} \quad r_n^+ := \min \{ r \geq 0 : \sigma_r \geq n\}.
\end{equation*}
The Lang--Weil estimate \eqref{Lang--Weil estimate} together with \eqref{key formula} therefore imply a sharp lower bound for $\mathbf{\mathcal N}(\vec{f}; p^{\alpha})$ whenever $1 \leq r_n^+ \leq n-2$. In the case $\vec{f} = \vec{P}$ is the power sum mapping, the condition $r_n^+ \leq n-2$ holds for all $n \geq 5$. On the other hand, the $n=1,2$ cases trivially admit sharp lower bounds whilst the remaining $n = 3, 4$ cases can be analysed via slightly more involved arguments (however, when $n=3$ some anomalies arise: see $\S$\ref{lower bounds section}). 

The issue of the minimum of $e(\vec{f}; \alpha, r)$ occurring when $r = n-1 \in R(\vec{f}; \alpha)$ does not arise if the number of polynomial equations $m$ satisfies $m\le n-2$. An important case is given by $m=1$ and $n \geq 3$, which treats a single non-degenerate (homogeneous) polynomial $f$ of, say, degree $d$. Here $\alpha_1 = \lceil\frac{\alpha}{d}\rceil \ge 1$ so that $R(f;\alpha) = \{0,1\}$. If $d \leq n$, then it is easy to see that $e(f;\alpha) = e(f;\alpha,1) = \alpha$. On the other hand, if $n \leq d$ and one writes $\alpha = (\alpha_1-1) d +j$ for some $1\leq j \leq d$, then 
\begin{equation*}
e(f;\alpha) = \left\{
\begin{array}{ll}
e(f;\alpha,1) = (n/d)\alpha + j(1-n/d) & \textrm{if $j\leq n \leq d$} \\
e(f;\alpha, 0) = n \alpha_1 & \textrm{if $n\le j \le d$}
\end{array}\right. . 
\end{equation*}
Hence Proposition \ref{key formula proposition} shows that, for $p$ sufficiently large depending on $d$,
\begin{equation*}
{\mathcal N}(f; p^{\alpha}) \sim_d \left\{\begin{array}{ll}
 p^{-\alpha}  & \textrm{if $d <n$} \\
\alpha p^{-\alpha} & \textrm{if $d=n$} \\
p^{-((n/d)\alpha - j(n/d -1))} & \textrm{if $j\le n < d$} \\ 
p^{-n\alpha_1} & \textrm{if $n<d$ and $n\le j \le d$}
\end{array}\right. .
\end{equation*}
This provides a precise count of the number of roots of $f$ over $\Z/p^{\alpha}\Z$. 

\begin{proof}[Proof (of Proposition \ref{key formula proposition})] The solution count 
$\mathbf{\mathcal N}(\vec{f}; p^{\alpha})$ can be expressed as the Haar measure of a set over the $p$-adics via a simple lifting procedure. In particular, for $f \in \Z[X_1, \dots, X_{n}]$ and $l \in \N_0$ define the $p$-adic sub-level set
\begin{equation*}
S(f,p^{-l}) := \big\{\vec{z} \in \Z_p^{n} : |f(\vec{z}\,)| \leq p^{-l}\big\},
\end{equation*} 
where $|\,\cdot\,|$ is the usual $p$-adic absolute value on $\Z_p$. It then follows that
\begin{equation}\label{N estimate 1}
\mathbf{\mathcal N}(\vec{f};  p^{\alpha}) = \mu\bigg(\bigcap_{k=1}^{m} S(f_k, p^{-\alpha})\bigg),
\end{equation}
where $\mu$ denotes the (normalised) Haar measure on the compact abelian group $\Z_p^{n}$. The space $\Z_p^{n}$ is foliated into countably many concentric annuli\footnote{The basic decomposition found in the work \cite{Denef2001} of Denef and Sperber is a foliation with respect to $p$-adic rectangles or boxes instead of concentric annuli. The notion of {\emph{non-degeneracy with respect to its Newton diagram}} emerges naturally from such a decomposition whereas the present notion of non-degeneracy arises naturally from a concentric annuli decomposition.} 
\begin{equation*}
\mathbb{A}_{p^{-l}}^{n}:= \big\{ \vec{z} \in \Z_p^{n} : |\vec{z}\,| := \max\{|z_1|, \dots, |z_{n}|\} = p^{-l} \big\}
\end{equation*}
so that 
\begin{align}
\nonumber
\mu\bigg(\bigcap_{k=1}^{m} S(f_k, p^{-\alpha})\bigg) &= \sum_{l \geq 0}  \mu|_{\mathbb{A}_{p^{-l}}^{n}}\bigg(\bigcap_{k=1}^{m} S(f_k, p^{-\alpha})\bigg) \\
\label{diagonal terms}
&= \sum_{l \geq 0} p^{-l n}\mu|_{\mathbb{A}_1^{n}}\bigg(\bigcap_{k=1}^{m} S(f_k, p^{-\alpha+l d_k })\bigg).
\end{align}
Recall that $\alpha_r$ for $1 \leq r \leq m+1$ forms a sequence of non-increasing, non-negative integers. One may therefore write
\begin{equation}\label{N estimate 2}
\mu\bigg(\bigcap_{k=1}^{m} S(f_k, p^{-\alpha})\bigg) = \sum_{r=0}^{m} I_r = \sum_{r \in R(\vec{f}; \alpha)} I_r
\end{equation}
where each $I_r$ is of exactly the same form as the expression appearing on the right-hand side of \eqref{diagonal terms} but with the summation in $l$ restricted to the range $\alpha_{r+1} \leq l < \alpha_r$. If $m \leq n-1$, then $R(\vec{f}; \alpha)$ is precisely the set of indices $0 \leq r \leq m$ for which the corresponding range of summation in $I_r$ is non-empty. If $m = n$, then using Lemma \ref{trivial solution lemma} it is not difficult to see that $I_n = 0$, and thus the second equality in \eqref{N estimate 2} is justified. 

If $\alpha_{r+1} \leq l$ and $r +1 \leq k \leq m$, then $\alpha - l d_k \leq 0$ and, consequently, $S(f_k, p^{-\alpha +l d_k}) = \Z_p^{n}$. On the other hand, if $l < \alpha_r$ and $1 \leq k \leq r$, then $\alpha - l d_k \geq 1$. Combining these observations,
\begin{equation*}
I_{0} = \sum_{ \alpha_1  \leq l} p^{-l n}  \sim p^{-e(\vec{f};\alpha, 0)}
\end{equation*}
and
\begin{equation*}
I_r = \sum_{\alpha_{r+1} \leq l <  \alpha_r} p^{-l n} \mu|_{\mathbb{A}_1^{n}}\bigg(\bigcap_{k=1}^{r} S(f_k, p^{-\alpha+l d_k})\bigg)
\end{equation*}
for $1 \leq r \leq m$, where each of the exponents $\alpha - l d_k$ appearing in the above expression is at least 1.   

\begin{lemma}\label{main estimate lemma} The identity
\begin{equation}\label{main estimate 1}
\mu|_{\mathbb{A}_1^{n}}\bigg(\bigcap_{k=1}^{r} S(f_k, p^{-l_k})\bigg) =  \big(\mathbf{\mathcal N}(f_1, \dots, f_r;p) - p^{-n}\big) \cdot p^{-\sum_{k=1}^{r} l_k + r}
\end{equation}
holds for all $1 \leq r \leq m$ and $l_1 \geq \dots \geq l_r \geq 1$.
\end{lemma} 

Temporarily assuming this lemma and letting $1 \leq r \leq m$ with $r \in R(\vec{f}; \alpha)$, one obtains the identity
\begin{equation*}
I_r = \big(\mathbf{\mathcal N}(f_1, \dots, f_r;p) - p^{-n} \big) \cdot p^{-r\alpha+r} \sum_{\alpha_{r+1} \leq l < \alpha_{r}} p^{-l(n-\sigma_{r})}.
\end{equation*}
A simple computation shows that
\begin{equation*}
\sum_{\alpha_{r+1} \leq l < \alpha_{r}} p^{-l(n-\sigma_{r})} \sim c_{n}(\alpha,r) \cdot \left\{\begin{array}{ll} 
p^{-(n-\sigma_{r})\alpha_{r+1}} & \textrm{if $\sigma_r \leq n$} \\
p^{-(n-\sigma_{r})(\alpha_{r} - 1)}  & \textrm{if $\sigma_r \geq n$}
\end{array}\right. 
\end{equation*}
where $c_{n}(\alpha,r) = \alpha_r - \alpha_{r+1}$ if $\sigma_r = n$ and $1$ otherwise. Substituting these estimates into the formula for $I_r$, it follows that 
\begin{equation}\label{N estimate 3}
 I_r \sim c_{n}(\alpha,r) \cdot \big(\mathbf{\mathcal N}(f_1, \dots, f_r; p) - p^{-n}\big) \cdot p^{-e(\vec{f}; \alpha, r) + r }
\end{equation}
holds for the exponent $e(\vec{f}; \alpha, r)$ from the introduction. Combining \eqref{N estimate 1}, \eqref{N estimate 2} and \eqref{N estimate 3}, one obtains the desired formula.
\end{proof}




\begin{proof}[Proof (of Lemma \ref{main estimate lemma})] The proof proceeds by inducting on $l_1$. If $l_1 = 1$, then $l_1 = \dots = l_r = 1$ and the left-hand side of \eqref{main estimate 1} can be written as 
\begin{equation*}
p^{-n}\#\big\{ \vec{z} \in \mathbb{F}_p^{n} \setminus \{\vec{0}\} : f_k(\vec{z}\,) \equiv 0 \bmod p \textrm{ for $1\leq k \leq r$}\big\}= \mathbf{\mathcal N}(f_1,\dots,f_r; p)-p^{-n},
\end{equation*}
as required. 

Now suppose $l_1 \geq 2$ and that the estimate is valid for all smaller values of $l_1$. The left-hand side of \eqref{main estimate 1} can be expressed as
\begin{equation}\label{main estimate 2}
\sum_{\vec{u} \in [\Z/p^{l_1-1}\Z]^{n} : |\vec{u}| = 1} p^{-(l_1-1)n}\mu\bigg(\bigcap_{k=1}^{r} S(f_k(\vec{u} + p^{l_1-1}\,\cdot\,), p^{-l_k})\bigg).
\end{equation}
For $1 \leq k \leq r$ one has
\begin{equation}\label{Taylor identity}
f_k(\vec{u} + p^{l_1-1}\vec{z}\,) \equiv f_k(\vec{u}\,) + p^{l_1-1} \nabla f_k(\vec{u}\,) \cdot \vec{z} \mod p^{l_k},
\end{equation}
which follows from Taylor's theorem and the fact that $2(l_1 - 1) \geq l_1 \geq l_k$. Let $1 \leq s \leq r$ be the largest integer such that $l_1 = \dots = l_s$. 
As a consequence of the identity \eqref{Taylor identity}, if $1 \leq k \leq s$ and $|f_k(\vec{u} + p^{l_1-1}\vec{z}\,)| \leq p^{-l_k}$ for some $\vec{z} \in \Z_p^{n-1}$, then $|f_k(\vec{u}\,)| \leq p^{-(l_k - 1)}$. On the other hand, if $s+1 \leq k \leq r$, then $f_k(\vec{u} + p^{l_1-1}\vec{z}\,) \equiv f_k(\vec{u}\,) \mod p^{l_k}$ for all  $\vec{z} \in \Z_p^{n-1}$. Define
\begin{equation*}
\tilde{l}_k := \left\{ \begin{array}{ll}
l_k - 1 & \textrm{for $1 \leq k \leq s$} \\
l_k & \textrm{for $s+1 \leq k \leq r$}
\end{array}\right. ,
\end{equation*}
noting that $\tilde{l}_k \geq 1$ for $1 \leq k \leq r$. For $\vec{u} \in \Z_p^{n}$ with $|f_k(\vec{u}\,)| \leq p^{-\tilde{l}_1}$ let $L_{k, \vec{u}}$ denote the degree 1 polynomial
\begin{equation*}
L_{k, \vec{u}}(\vec{z}\,) := p^{-\tilde{l}_1}f_k(\vec{u}\,) +  \nabla f_k(\vec{u}\,) \cdot \vec{z}.
\end{equation*}
Combining the above observations, \eqref{main estimate 2} can be written as 
\begin{equation}\label{main estimate 3}
\sum_{\substack{\vec{u} \in [\Z/p^{\tilde{l}_1}\Z]^{n} : |\vec{u}| = 1 \\ |f_k(\vec{u}\,)| \leq p^{-\tilde{l}_k} \textrm{for $1 \leq k \leq r$} }} p^{-\tilde{l}_1n} \mu\bigg(\bigcap_{k=1}^{s} S(L_{k, \vec{u}}, p^{-1})\bigg).
\end{equation}
Thus the problem is reduced to estimating the size of intersections of neighbourhoods of certain hyperplanes in $\Z_p^{n}$. 

\begin{lemma}\label{intersecting planes lemma} If $\vec{u} \in \mathbb{A}_1^{n}$ satisfies $|f_k(\vec{u}\,)| \leq p^{-1}$ for  $1 \leq k \leq r$, then
\begin{equation}\label{intersecting planes lemma 1}
\mu\bigg(\bigcap_{k=1}^{s} S(L_{k, \vec{u}}, p^{-1})\bigg)  = p^{-s}
\end{equation}
for $1 \leq s \leq r$.
 \end{lemma}

Applying this identity to \eqref{main estimate 3} one observes that \eqref{main estimate 2} can be written as
\begin{equation*}
p^{-s-\tilde{l}_1n}\#\big\{\vec{u} \in [\Z/p^{\tilde{l}_1}\Z]^n : |\vec{u}\,| = 1 \textrm{ and } f_k(\vec{u}\,) \equiv 0 \bmod p^{\tilde{l}_k} \textrm{ for $1 \leq k \leq r$} \big\}, 
\end{equation*}
which can then be expressed as the $p$-adic integral
\begin{equation*}
p^{-s} \mu|_{\mathbb{A}_1^{n}}\bigg(\bigcap_{k=1}^{r} S(f_k, p^{-\tilde{l}_k})\bigg).
\end{equation*}
The induction hypothesis, coupled with the identity
\begin{equation*}
s+ \sum_{k=1}^{r} \tilde{l}_k = \sum_{k=1}^{r} l_k,
\end{equation*}
now implies that
\begin{equation*}
p^{-s} \mu|_{\mathbb{A}_1^{n}}\bigg(\bigcap_{k=1}^{r} S(f_k, p^{-\tilde{l}_k})\bigg) 
= \big(\mathbf{\mathcal N}(f_1, \dots, f_r; p)-p^{-n}\big)\cdot p^{-\sum_{k=1}^{r} l_k + r}
\end{equation*}
Combining the preceding chain of identities closes the induction and concludes the proof of Lemma \ref{main estimate lemma}. 
\end{proof}




It remains to prove Lemma \ref{intersecting planes lemma}.

\begin{proof}[Proof (of Lemma \ref{intersecting planes lemma})] Given $\vec{u} \in \Z_p^n$ satisfying the hypotheses of the lemma, one wishes to study the intersection properties of the sets
\begin{equation*}
 S(L_{\vec{u}, k}, p^{-1}) = \big\{\vec{z} \in \Z_p^{n} : |p^{-l_1+1}f_k(\vec{u}) +  \nabla f_k(\vec{u}) \cdot \vec{z}\, | \leq p^{-1}\big\}.
\end{equation*}
The non-degeneracy hypothesis implies that $\nabla f_k(\vec{u}\,) \not\equiv 0 \bmod p$ for $1 \leq k \leq r$, and so each $ S(L_{\vec{u}, k}, p^{-1})$ is a $p^{-1}$-neighbourhood of a hyperplane in $\Z_p^{n}$. The size of the intersection of the $S(L_{\vec{u}, k}, p^{-1})$ for $1 \leq k \leq s$ is therefore governed by angles between the normal vectors $\nabla f_1(\vec{u}\,), \dots, \nabla f_{s}(\vec{u}\,)$. More precisely, the non-degeneracy hypothesis implies the existence of some $\mathbf{j} = (j_1, \dots, j_{n-s}) \in \{1, \dots, n\}^{n-s}$ such that
\begin{equation}\label{transverse planes}
 |\det\begin{pmatrix}
       \nabla f_1(\vec{u}\,) & \dots & \nabla f_{s}(\vec{u}\,) & e_{j_1} & \dots & e_{j_{n-s}}
      \end{pmatrix}| = 1,
\end{equation}
where the $e_j$ are the standard basis vectors. Expressing the left-hand side of \eqref{intersecting planes lemma 1} as
\begin{equation*}
\int_{\Z_p^{n}} \prod_{k=1}^{s} \chi_{B(0,p^{-1})}(p^{-l_1+1}f_k(\vec{u}) +  \nabla f_k(\vec{u}\,) \cdot \vec{z} )\,\ud \mu (\vec{z}\,),
\end{equation*}
it follows from the $p$-adic change of variables formula (see, for instance, \cite[$\S$7.5]{Igusa2000}), that 
\begin{equation*}
 \mu\bigg(\bigcap_{k=1}^{s} S(L_{k, \vec{u}}, p^{-1})\bigg) = \int_{\Z_p^s} \prod_{k=1}^{s} \chi_{B(0,p^{-1})}(y_j)\,\ud \mu (\vec{y}\,) = p^{-s},
\end{equation*}
as required. 
\end{proof}




\section{Applying Theorem \ref{general theorem} to count factorisations of monomials}\label{application section}

In order to apply Theorem \ref{general theorem} to the problem of factorising monomials, one must verify that the system of power sums satisfies the non-degeneracy hypothesis. 

\begin{lemma} If $P_k \in \Z[X_1, \dots, X_n]$ denotes the $k$-th power sum, then $\vec{P} := (P_1, \dots, P_n)$ is non-degenerate over $\mathbb{F}_p$ for all primes $p > n$. 
\end{lemma}

\begin{proof} Fixing $1 \leq r \leq n$, suppose that $\vec{z} \in \overline{\mathbb{F}}_p^n$ satisfies 
\begin{equation*}
P_k(\vec{z}\,) = 0 \qquad \textrm{for $1 \leq k \leq r$}
\end{equation*}
and that the vectors $\nabla P_1(\vec{z}\,), \dots, \nabla P_r(\vec{z}\,)$ are linearly dependent in $\overline{\mathbb{F}}_p^n$. 
If $\mathbf{k} = (k_1, \dots, k_{r}) \in \N^{r}$ with $1 \leq k_1 < \dots < k_{r} \leq n$, then it follows that
\begin{equation}\label{linearly dependent}
\det \begin{pmatrix}
\nabla^{\mathbf{k}}P_1 (\vec{z}\,) & \dots & \nabla^{\mathbf{k}}P_{r} (\vec{z}\,) \\
\end{pmatrix} = 0,
\end{equation}
where $\nabla^{\mathbf{k}} \colon \Z[X_1, \dots, X_n] \mapsto \Z[X_1, \dots, X_n]^{r}$ is the differential operator
\begin{equation*}
\nabla^{\mathbf{k}}f := (\partial_{x_{k_1}} f, \dots, \partial_{x_{k_{r}}} f).
\end{equation*}
The determinant in \eqref{linearly dependent} is given by
\begin{equation*}
r!\prod_{1 \leq i < j \leq r} (z_{k_i} - z_{k_j})
\end{equation*}
where, by the hypothesis $p > n$, $r! \not\equiv 0 \bmod p$. Combining these observations, it follows that the $z_1, \dots, z_n$ assume at most $r-1$ values in $\overline{\mathbb{F}}_p^n$. In particular, there exists a partition $A_1, \dots, A_t$ of $\{1, \dots, n\}$ into at most $r-1$ non-empty sets and a collection of distinct elements $x_i \in \overline{\mathbb{F}}_p$ such that
\begin{equation*}
z_k = x_i \quad \textrm{whenever $k \in A_i$.}
\end{equation*}
Thus, if $a_i = \#A_i$, then $(x_1, \dots, x_t) \in \overline{\mathbb{F}}_p^t$ is a solution to the square system
\begin{equation}\label{weighted system 0}
\sum_{i = 1}^t a_i X_i^k = 0 \quad \textrm{for $1 \leq k \leq t$.}
\end{equation}
For any non-empty $S \subseteq \{1, \dots, t\}$ it follows that $1 \leq \sum_{i \in S} a_i \leq n$ and so $p \nmid \sum_{i \in S} a_i$. It is shown in Proposition \ref{weighted system proposition} below that, under these hypotheses, only the trivial solution satisfies a system of the form \eqref{weighted system 0}, and one therefore deduces that $z_k = 0$ for $1 \leq k \leq n$. This shows the system $\vec{P} = (P_1,\ldots, P_n)$ is non-degenerate over ${\mathbb F}_p$.
\end{proof}

The above argument relied upon the following proposition.

\begin{proposition}\label{weighted system proposition} Let $R$ be an integral domain, $\vec{a} = (a_1, \dots, a_n) \in R^n$ and define the weighted power sums
\begin{equation*}
 P_{\vec{a},k}(X_1, \dots, X_n) := \sum_{i = 1}^n a_i X_i^k \qquad \textrm{for $1\leq k \leq n$.}
\end{equation*}
If $\sum_{i \in S} a_i \neq 0$ for all non-empty $S \subseteq \{1, \dots, n\}$, then the system 
\begin{equation}\label{weighted system}
P_{\vec{a},k}(X_1, \dots, X_n) = 0 \qquad \textrm{for $1\leq k \leq n$}
\end{equation}
has a unique (trivial) solution in $R^n$. 
\end{proposition}

The proposition is a consequence of the following identity (see \cite[Theorem 4.3]{Kumar2012} for an alternative approach).

\begin{lemma}[Weighted Newton--Girard formula]\label{weighted Newton Girard lemma} If $R$ is commutative ring (with identity) and $\vec{a} = (a_1, \dots, a_n) \in R^n$, then
\begin{equation}\label{weighted Newton Girard}
 \big(\sum_{i=1}^n a_i\big)e_n(X_1, \dots, X_n) = \sum_{k=1}^n (-1)^{k-1}e_{n-k}(X_1, \dots, X_n)P_{\vec{a}, k}(X_1, \dots, X_n),
\end{equation}
where the $e_k \in R[X_1, \dots, X_n]$ are the elementary symmetric polynomials in $n$ variables. 
\end{lemma}

\begin{proof}
Observe that
\begin{equation*}
 e_{n-k}(X_1, \dots, X_n) = X_i e_{n-(k+1)}(X_1,\dots \hat{X}_i \dots, X_n) + e_{n-k}(X_1,\dots \hat{X}_i \dots, X_n)
\end{equation*}
for $1 \leq i, k \leq n$, where the notation $\hat{X}_i$ is used to signify the omission of the $X_i$ variable. Here 
$e_{-1}$ is interpreted as the zero polynomial (and $e_0$ is the constant polynomial 1). Using this identity, one may express the right-hand side of \eqref{weighted Newton Girard} as
\begin{equation*}
 \sum_{i=1}^n a_i \sum_{k=1}^n (-1)^{k-1} \big(X_i^{k+1}e_{n-(k+1)}(X_1,\dots \hat{X}_i \dots, X_n) + X_i^k e_{n-k}(X_1,\dots \hat{X}_i \dots, X_n)\big).
\end{equation*}
Each sum in the $k$ index is telescoping and it is easy to see that the above expression reduces to 
\begin{equation*}
  \sum_{i=1}^n a_i X_i e_{n-1}(X_1,\dots \hat{X}_i \dots, X_n) = \big(  \sum_{i=1}^n a_i\big) e_n(X_1, \dots, X_n),
\end{equation*}
as required. 
\end{proof}

The proposition is now immediate. 

\begin{proof}[Proof (of Proposition \ref{weighted system proposition})] The proof is by induction on $n$, the case $n=1$ being vacuous. Suppose that $\vec{x} = (x_1,\dots, x_n) \in R^n$ is a solution to \eqref{weighted system}. By Lemma \ref{weighted Newton Girard lemma} one has
\begin{equation*}
\big(\sum_{i=1}^n a_i\big)x_1\dots x_n = 0 
\end{equation*}
and, since by hypothesis $\sum_{i=1}^n a_i \neq 0$ and $R$ is an integral domain, one deduces that $x_i =0$ for some $1 \leq i \leq n$. Without loss of generality suppose that $x_n=0$. Then $(x_1, \dots, x_{n-1}) \in R^{n-1}$ is a solution to the system
\begin{equation*}
P_{\vec{a}',k}(X_1, \dots, X_{n-1}) = 0  \qquad \textrm{for $1\leq k \leq n-1$}
\end{equation*}
where the coefficient vector $\vec{a}' := (a_1, \dots, a_{n-1})$ automatically satisfies the hypothesis of the proposition. Thus, by the induction hypothesis, $x_i= 0$ for $1 \leq i \leq n-1$, as required.  
\end{proof}




\section{Lower bounds in Theorem \ref{main theorem}}\label{lower bounds section}

As already noted, the upper bounds in Theorem \ref{main theorem} follow from Theorem \ref{general theorem}. A slightly more refined analysis is needed to complete the proof of the asymptotic formula \eqref{main theorem estimate} in Theorem \ref{main theorem} for all degrees $n \neq 3$. In particular, it is now shown that for all sufficiently large primes $p$ the inequality
\begin{equation}\label{tight lower bound}
\mathbf{N}(\vec{0}_n; p^{\alpha}) \gg_{n} [\alpha]^{\delta_n(\triangle)} p^{-e_n(\alpha)} 
\end{equation}
holds for all $\alpha \in \N$ and $n\not= 3$. As by-product of analysis, the sharp result in the $n=3$ case may also be derived, which curiously has additional dependence on both the parity of $\alpha$ and the congruence class of $p$ modulo $3$. 

Recall the key formula
\begin{equation}\label{key formula reproduced}
\mathbf{N}(\vec{0}_n; p^{\alpha}) \sim_n \sum_{r \in R_n(\alpha)} c_n(\alpha,r) \cdot 
\big(\mathbf{\mathcal N}(P_1, \dots, P_r;p) - p^{-n}\big) \cdot p^{-e_n(\alpha, r) +r}
\end{equation}
established in Proposition \ref{key formula proposition}. First suppose that $n = \triangle_{r'}$ for some $r' \geq 2$ and that $r' \in R_n(\alpha)$ so that
$ [\alpha]^{\delta_n(\triangle)} = \alpha$. In this case, $c_n(\alpha,r') \sim_n \alpha$ and $c_n(\alpha,r) = 1$ for all other values of $r$. If $n \neq 3$, then $r' \leq n-2$ and so the Lang--Weil bound \eqref{Lang--Weil estimate} yields
\begin{equation}\label{tight lower bound 1}
\mathbf{\mathcal N}(P_1, \dots, P_{r};p) \gg_n p^{-r}.
\end{equation}
for $r= r'$. The lower bound \eqref{tight lower bound} now follows by combining \eqref{tight lower bound 1} with \eqref{key formula reproduced}. Thus, provided $n \neq 3$, one may assume without loss of generality that   $[\alpha]^{\delta_n(\triangle)} = 1$ and $c_n(\alpha,r) = 1$ for all $\alpha$ and $r$.

Focusing on the $n \neq 3$ case, it now suffices to show that for $n \neq 3$ the estimate
\begin{equation}\label{tight lower bound 3}
\mathbf{N}(\vec{0}_n; p^{\alpha}) \gg_{n} p^{-e_n(\alpha)} 
\end{equation}
holds for all $\alpha \in \N$. By \eqref{key formula reproduced}, this would follow if one could demonstrate that there exists some $r \in R_n(\alpha)$ for which $e_n(\alpha) = e_n(\alpha, r)$ and \eqref{tight lower bound 1} holds. 

\subsection*{Low degree case} For any value of $n \in \N$ it is immediate that  
\begin{equation}\label{trivial finite field estimates}
\mathbf{\mathcal N}(\emptyset ; p) = 1 \quad \textrm{and} \quad \mathbf{\mathcal N}(P_1 ; p) = p^{-1}.
\end{equation}
From these identities it follows that \eqref{tight lower bound} holds, establishing Theorem \ref{main theorem} for $n=1,2$. 

\subsection*{High degree case} Here \eqref{tight lower bound 3} is established for degrees $n \geq 4$. Recall that the Lang--Weil estimate \eqref{Lang--Weil estimate} implies that \eqref{tight lower bound 1} holds for all $1 \leq r \leq n-2$ and, trivially, \eqref{tight lower bound 1} is also valid for $r = n$. Thus, if either $n-1 \notin R_{n}(\alpha)$ or there exists some $1 \leq r \leq n-2$ with $r \in R_n(\alpha)$ and $e_n(\alpha, n-1) \geq e_n(\alpha, r)$, then the desired lower bound \eqref{tight lower bound 3} immediately follows. 

These observations allow one to easily treat the $n= 4$ case. It is useful to first compute the relevant exponents:
\begin{equation*}
\begin{array}{rclcrcl}
e_4(\alpha, 0) &=& 4\alpha,                                          &\qquad&  e_4(\alpha, 1) &=& \alpha + 3 \lceil \frac{\alpha}{2} \rceil, \\[3pt]
e_4(\alpha, 2) &=& 2\alpha + \lceil \frac{\alpha}{3} \rceil,         &\qquad&  e_4(\alpha, 3) &=& 3 \alpha - 2 \lceil \frac{\alpha}{3} \rceil + 2.
\end{array}
\end{equation*}
Since $e_4(\alpha, 2) \leq e_4(\alpha, 3)$, the lower bound \eqref{tight lower bound 3} holds whenever $2 \in R_4(\alpha)$ or $3 \notin R_4(\alpha)$. This leaves only $\alpha = 4$, but since $e_4(4, 1) = e_4(4, 3) = 10$, the result also holds in this case. 

A similar, but more involved, argument allows one to treat $n \geq 5$. Let $n \geq 5$ and suppose, aiming for a contradiction, that $n-1 \in R_{n}(\alpha)$ and that there exists no value of $r \in R_n(\alpha)$ with $1 \leq r \leq n-2$ and $e_n(\alpha, n-1) \geq e_n(\alpha, r)$. Recall from the introduction that 
\begin{equation*}
e_n(\alpha, r+1) - e_n(\alpha, r) \geq 1 \quad \textrm{for $\triangle_r \geq n$.}
\end{equation*} 
It therefore follows that $r \notin R_n(\alpha)$ for $r_n^+ \leq r \leq n-2$ and so 
\begin{equation*}
k := \Big\lceil \frac{\alpha}{r_n^+} \Big\rceil = \Big\lceil \frac{\alpha}{n-1} \Big\rceil. 
\end{equation*}
In particular, 
\begin{equation}\label{comparing terms 1}
\alpha = (k-1)(n-1) + j = (k-1)r_n^+ + j'
\end{equation}
for some $1 \leq j \leq n-1$ and $1 \leq j' \leq r_n^+$. 

If $k = 1$, then $\lceil \tfrac{\alpha}{n-1} \rceil = 1 = \lceil \tfrac{\alpha}{n}\rceil$, contradicting the assumption that $n-1 \in R_n(\alpha)$. If $j' = r_n^+$, then 
\begin{equation*}
 \frac{\alpha}{r_n^+} = \Big\lceil \frac{\alpha}{r_n^+} \Big\rceil < \Big\lceil \frac{\alpha}{r_n^+-1} \Big\rceil
\end{equation*}
and so $r_n^+ - 1 \in R_n(\alpha)$. Furthermore,
\begin{align*}
e_n(\alpha, r_n^+) - e_n(\alpha, r_n^+ - 1) &= \alpha + (n-\triangle_{r_n^+})(\lceil \tfrac{\alpha}{r_n^+} \rceil - 1) -  (n-\triangle_{r_n^+-1})\lceil \tfrac{\alpha}{r_n^+} \rceil \\
&= \alpha - r_n^+\cdot\lceil \tfrac{\alpha}{r_n^+} \rceil  + (\triangle_{r_n^+}-n) = \triangle_{r_n^+}-n \geq 0,
\end{align*}
which implies that $e_n(\alpha, n-1) \geq e_n(\alpha, r_n^+ - 1)$, a contradiction. 

The above observations show that $k\geq 2$ and $1 \leq j' \leq r_n^+ - 1$ and so \eqref{comparing terms 1} yields
\begin{equation*}
(k-1)(n-1-r_n^+) = j' - j \leq r_n^+ - 2.
\end{equation*}
Write $n =  \triangle_{r_n^+-1} + s_n$ for some $1 \leq s_n \leq r_n^+$ so that
\begin{equation*}
(k-1)(r_n^+(r_n^+-3) + 2(s_n-1)) \leq 2r_n^+ - 4.
\end{equation*}

If $k\geq 3$ and  $s_n \geq 2$, then $r_n^+$ must satisfy 
\begin{equation*}
(x-2)^2 = x^2-4x + 4 \leq 0
\end{equation*}
which forces $r_n^+ = 2$ and hence $n=3$, but this contradicts the assumption $n\geq 5$. Furthermore, if $k\geq 3$ and $s_n =1$, then $r_n^+$ must satisfy $x^2 - 4x +2 \leq 0$. This implies that $r_n^+ \in \{2, 3\}$ and so $3 \leq n \leq 6$; of these values of $n$, only $n=4$ satisfies $s_n = 1$, again contradicting the assumption that $n\geq 5$.

From the preceding analysis one deduces that $k = 2$ and that $n$ must satisfy the inequality
\begin{equation}\label{comparing terms 2}
t_n := (r_n^+)^2 -5r_n^+ +2(s_n+1)\leq 0.
\end{equation}
Again using the basic estimate $s_n \geq 1$, it follows that $r_n^+ \leq 4$ and so $5 \leq n \leq 10$. Furthermore, an explicit computation (see Figure \ref{lots of values}) now shows that $5$ and $7$ are the only values of $n$ for which \eqref{comparing terms 2} holds. For both $n=5$ and $n=7$ the inequality \eqref{comparing terms 2} is saturated. Consequently, $j$ must assume the extreme value $j = 1$ and so $\alpha = n$. Finally, by direct computation one may show that $2 \in R_5(5)$, $3 \in R_7(7)$ and 
\begin{equation*}
e_5(5, 2) = 14 \leq e_5(5,4) = 15 \quad e_7(7, 3) = 23 \leq e_7(7,6) = 28,
\end{equation*}
which is the desired contradiction.

\begin{figure}
\begin{center}
\begin{tabular}{ |c|c|c|c| } 
 \hline
$n$ & $r_n^+$ & $s_n$ & $t_n$   \\[6pt]
\hline
5   &    3    &   2   &        0                       \\
6   &    3    &   3   &        2                       \\
7   &    4    &   1   &        0                       \\
8   &    4    &   2   &        2                       \\
9   &    4    &   3   &        4                       \\
10  &    4    &   4   &        6                       \\
 \hline 
\end{tabular}
\caption{If $n \geq 5$, then $t_n \leq 0$ holds only for $n=5$ and $n=7$.}
\label{lots of values}
\end{center}
\end{figure}

\subsection*{Intermediate degree ($n=3$) case} It remains to examine the situation when $n=3$, which is a little more complicated. 

Given $\vec{a} = (a_1, \dots, a_r) \in \Z$ define $f_{\vec{a}}(x) := \sum_{j=1}^r a_j x^j$. For any $n \in \N$ the basic properties of character sums imply that
\begin{align}
\nonumber
\mathbf{\mathcal N}(P_1, \dots, P_r; p) &= p^{-n-r} \sum_{a_1  \; (p)} \dots \sum_{a_r  \; (p)}  \Big(\sum_{x \; (p)} e^{2 \pi i f_{\vec{a}}(x)/p}\Big)^n \\
\label{exponential sum}
&= p^{-r} + p^{-r}\sum_{\vec{a}  \; (p) \,:\, \vec{a} \not\equiv \vec{0} \bmod p} \Big(p^{-1}\sum_{x \; (p)}  e^{2 \pi i f_{\vec{a}}(x)/p}\Big)^n, 
\end{align}
where the sums in $x$ and the $a_j$ in the first line are each over a complete set of residues modulo $p$. If $r = 2$ and $p$ is odd, then the above expression involves classical Gauss sums which can be evaluated using the formula
\begin{equation}\label{Gauss formula}
p^{-1}\sum_{x\; (p)} e^{2\pi i (ax^2 + bx)/p} = \varepsilon_p(a| p)p^{-1/2}e^{-2\pi i (4a)^{-1}b^2/p} \qquad \textrm{for $a \not\equiv 0 \bmod p$}.
\end{equation}
Here $\varepsilon_p = 1$ whenever $p \equiv 1 \bmod 4$ and $\varepsilon_p = i$ otherwise, and $(a| p)$ is the Legendre symbol. Indeed, by completing the square in the phase, \eqref{Gauss formula} is a direct consequence of Gauss' classical formula for quadratic Gauss sums (see, for instance, \cite[$\S$9.10]{Apostol1976}). Writing $\mathbf{\mathcal N}(P_1, P_2; p) = p^{-2} + E$, it follows from the above identity that
\begin{equation*}
E = \varepsilon_p^n p^{-(n+4)/2} \sum_{a \;(p) \,:\, a \not\equiv 0 \bmod p}(a| p)^n\sum_{b \; (p)}e^{-2\pi i n (4a)^{-1}b^2/p}.
\end{equation*}
The sum in $b$ can also be evaluated and, applying elementary properties of quadratic residues (in particular, the completely multiplicative property of the Legendre symbol), one obtains
\begin{equation*}
E = \varepsilon_p^{n+1}(-n|p) p^{-(n+3)/2} \sum_{a\;(p)  : a \not\equiv 0 \bmod p}(a| p)^{n-1}.
\end{equation*}
Recall that there are precisely $(p-1)/2$ non-zero quadratic residues and $(p-1)/2$ quadratic non-residues modulo $p$. Thus,
\begin{equation*}
\sum_{a\;(p)  : a \not\equiv 0 \bmod p}(a| p)^{n-1} = (1 + (-1)^{n-1})\cdot \frac{p-1}{2}
\end{equation*}
and, consequently, 
\begin{equation}\label{Gauss sum identity}
E = \left\{\begin{array}{ll}
0 & \textrm{if $n$ is even} \\
\varepsilon_p^{n+1}(-n|p) p^{-(n+3)/2}(p-1) & \textrm{if $n$ is odd}
\end{array}\right. . 
\end{equation}

The above formula can be used to treat the $n=3$ case, which behaves in a distinctly different manner from that of every other degree. Here the relevant exponents are given by
\begin{equation*}
e_3(\alpha, 0) = 3\alpha, \qquad
e_3(\alpha, 1) = \alpha + 2\lceil \frac{\alpha}{2} \rceil, \qquad
e_3(\alpha, 2) = 2\alpha.
\end{equation*}
If $2 \notin R_3(\alpha)$ (that is, $\alpha \in \{1,2,4\}$), then the bound
\begin{equation}\label{tight lower bound 4}
\mathbf{N}(\vec{0}_3; p^{\alpha}) \gg [\alpha]^{\delta_3(\triangle)}p^{-e_3(\alpha)}
\end{equation}
follows immediately from the trivial identities stated in \eqref{trivial finite field estimates} (note that, in this case, $[\alpha]^{\delta_n(\triangle)} = 1$). If $2 \in R_3(\alpha)$, then the analysis is more complex. The identity \eqref{Gauss sum identity} implies that\footnote{Here it is assumed that $p > 3$ so that $(-3|p) \neq 0$.}
\begin{equation*}
\mathbf{\mathcal N}(P_1, P_2; p) = \left\{ \begin{array}{ll} 
2p^{-2} - p^{-3} & \textrm{if $-3$ is a quadratic residue modulo $p$} \\
p^{-3} & \textrm{otherwise}
\end{array}\right. .
\end{equation*}
By the law of quadratic reciprocity, $-3$ is a quadratic residue modulo $p$ if and only if $p \equiv 1 \bmod 3$. Thus, if $2 \in R_3(\alpha)$ and $p \equiv 1 \bmod 3$, then \eqref{tight lower bound 4} once again holds. Now suppose that $2 \in R_3(\alpha)$ and $p \not\equiv 1 \bmod 3$. If $\alpha$ is even, then $e_3(\alpha ,1) = e_3(\alpha,2) = 2\alpha$ and so 
\begin{equation*}
\mathbf{N}(\vec{0}_3; p^{\alpha}) \sim p^{-e_3(\alpha)},
\end{equation*}
by \eqref{key formula reproduced}, which differs by a logarithm from what one would expect based on the bounds for $n \neq 3$. If $\alpha$ is odd, then $e_3(\alpha, 1) = e_3(\alpha, 2) + 1$ and so
\begin{equation*}
\mathbf{N}(\vec{0}_3; p^{\alpha}) \sim p^{-e_3(\alpha) -1}
\end{equation*}
again by \eqref{key formula reproduced}, which differs by a factor of $p^{-1}$ (up to a logarithmic factor) from what one would expect based on the bounds for $n \neq 3$. 
The situation for $n=3$ is therefore summarised as follows.

\begin{lemma}\label{cubic lemma} If $p$ is a sufficiently large prime, then 
\begin{equation*}
\mathbf{N}(\vec{0}_3; p^{\alpha}) \sim [\alpha]^{\delta_3(\triangle)\cdot \kappa(p)} p^{-e_{3}(\alpha) - \lambda(\alpha)\cdot(1-\kappa(p))}
\end{equation*}
holds for all $\alpha \in \N$ where
\begin{equation*}
\lambda(\alpha) := \left\{\begin{array}{ll}
1 & \textrm{if $\alpha > 1$ is odd} \\
0 & \textrm{otherwise}
\end{array}\right.\quad \textrm{and}\quad 
\kappa(p) :=  \left\{\begin{array}{ll}
1 & \textrm{if $p \equiv 1 \mod 3$} \\
  0    & \textrm{otherwise}
\end{array}\right. .
\end{equation*}
\end{lemma}

Thus, Lemma \ref{cubic lemma} and Theorem \ref{main theorem} provide a precise count of the number of factorisations of $X^n$ over $\Z/p^{\alpha}\Z$ for all degrees $n$, provided the prime $p$ is sufficiently large. 




\appendix

\section{Irreducibility of projective varieties defined by non-degenerate systems}\label{algebraic geometry section}

Recall that the Lang--Weil theorem \cite{Lang1954} was used to derive Proposition \ref{classical estimates proposition}, which formed the base case of the induction argument used to prove Theorem \ref{general theorem}. To justify the application of the Lang--Weil bound, one must verify that certain projective varieties defined by the homogenous polynomials $f_k$ are \emph{absolutely irreducible}: that is, they are irreducible as varieties over $\mathbb{P}^{n-1}(\overline{\mathbb{F}}_p)$. 

\begin{lemma}\label{irreducible lemma} Suppose $1 \leq m \leq n-2$ and $\vec{f} := (f_1, \dots, f_m)$ is a system of homogeneous polynomials that satisfies the non-degeneracy hypothesis. The projective variety 
\begin{equation}\label{irreducible lemma variety}
V := \{x \in \mathbb{P}^{n-1}(\overline{\mathbb{F}}_p) : f_j(x) = 0 \textrm{ for $1 \leq j \leq m$ }\}
\end{equation}
is irreducible. 
\end{lemma}

As a consequence of the proof of Lemma \ref{irreducible lemma}, one can also verify the dimension condition needed for the application of the Schwarz--Zippel bound in Lemma \ref{Schwarz--Zippel lemma}.

\begin{corollary}\label{dimension corollary} Suppose $1 \leq m \leq n-1$ and $\vec{f} := (f_1, \dots, f_m)$ is a system of homogeneous polynomials that satisfies the non-degeneracy hypothesis. If $V$ is as in \eqref{irreducible lemma variety}, then $\dim V = n - 1 -r$.  
\end{corollary}

Before stating the proof of Lemma \ref{irreducible lemma} and Corollary \ref{dimension corollary}, it is useful to review some of the basic concepts from commutative algebra and algebraic geometry which appear in the argument. All the facts and definitions presented below are standard and can be found in many textbooks (see, for instance, \cite{Kunz2013}). 

Let $K$ be an algebraically closed field and $R$ be a commutative, Noetherian ring (for instance, $R = K[X_1, \dots, X_n]$).
\begin{itemize}
\item A projective variety $V \subseteq \mathbb{P}^{n-1}(K)$ is the zero-locus of a set $f_1, \dots, f_r \in K[X_1, \dots, X_n]$ of homogeneous polynomials (note that here a variety is \emph{not} required to be irreducible). If, in particular, $V = \{x \in \mathbb{P}^{n-1}(K) : f(x) = 0 \}$ is the zero-locus of a single non-constant homogeneous polynomial $f \in K[X_1, \dots, X_n]$, then $V$ is said to be a projective hypersurface.  
\item The ideal $\mathcal{I}(V)$ of a variety $V \subseteq \mathbb{P}^{n-1}(K)$ is the collection of all polynomials in $K[X_1, \dots, X_n]$ which vanish on $V$. Fixing homogeneous polynomials $f_1, \dots, f_r \in K[X_1, \dots, X_n]$ and defining
\begin{equation*}
V := \{ x \in \mathbb{P}^{n-1}(K) : f_j(x) = 0 \textrm{ for all $1 \leq j \leq r$}\},
\end{equation*}
if $I := \langle f_1, \dots, f_r \rangle$ denotes the homogeneous ideal generated by the $f_1, \dots, f_r$, then Hilbert's Nullstellensatz states that
\begin{equation*}
\mathcal{I}(V) = \sqrt{I}.
\end{equation*}
Here, for any ideal $I \unlhd R$, the radical ideal $\sqrt{I}$ is defined by
\begin{equation*}
\sqrt{I} := \{f \in R : f^m \in I \textrm{ for some $m \in \N$ } \}
\end{equation*}
\item A projective variety $V \subseteq \mathbb{P}^{n-1}(K)$ is irreducible if the following holds: if $V = V_1 \cup V_2$ for $V_1, V_2 \subseteq \mathbb{P}^{n-1}$ projective varieties, then $V_1 = V$ or $V_2 = V$. This condition is equivalent to the primality of the ideal $\mathcal{I}(V)$.

\item A chain of prime ideals of the form $\mathfrak{p}_0 \subset \mathfrak{p}_1 \subset \dots \subset \mathfrak{p}_k$ is said to have length $k$ (here each $\mathfrak{p}_j \unlhd R$ is a prime ideal and the inclusions are strict). The Krull dimension of a ring $R$, which is denoted by $\dim R$, is the supremum of all lengths of chains of prime ideals in $R$. As a key example, $\dim K[X_1, \dots, X_n] = n$; in fact the length of \emph{any} maximal chain of prime ideals in $K[X_1, \dots, X_n]$ is $n$ (see, for instance, \cite[Chapter II, Proposition 3.4]{Kunz2013}).

\item Given a prime ideal $\mathfrak{p} \unlhd R$ define the height of $\mathfrak{p}$ to be the supremum of all lengths of prime ideals of $R$ contained in $\mathfrak{p}$. The height of an arbitrary (that is, not necessarily prime) proper ideal $I\unlhd R$, which is denoted $\mathrm{height}(I)$, is then defined to be the infimum of the heights of all prime ideals which contain $I$. Since any prime ideal containing $I$ automatically contains $\sqrt{I}$, it follows that $\mathrm{height}(\sqrt{I}) = \mathrm{height}(I)$. The generalised Krull principal ideal theorem (see, for instance, \cite[Chapter V, Theorem 3.4]{Kunz2013}) asserts that if $I = \langle f_1 , \dots, f_r \rangle$ is generated by $r$ elements, then $\mathrm{height}(I) \leq r$. 
\item The dimension $\dim V$ of a projective variety $V$ is given by $\dim V := \dim K[V] - 1$ where $K[V]$ is the \emph{co-ordinate ring}
\begin{equation*}
 K[V] := K[X_1, \dots, X_n]/\mathcal{I}(V)
\end{equation*}
(see, for instance, \cite[Chapter II, Proposition 4.4]{Kunz2013}). As a consequence of the correspondence theorem for prime ideals, it follows that $\dim V \leq \dim K[X_1, \dots, X_n] - 1  - \mathrm{height}(\mathcal{I}(V))$. In fact, since the length of any maximal chain of prime ideals in $K[X_1,\dots, X_n]$ is $n$ (see, for instance, \cite[Chapter II Proposition 3.4]{Kunz2013}), it is not difficult to see that equality holds; that is,
\begin{equation}\label{dimension formula}
\dim V = n - 1 - \mathrm{height} (\mathcal{I}(V)).
\end{equation}
\item A projective variety $V \subseteq \mathbb{P}^n(K)$ is a set-theoretic complete intersection if it is the intersection of $n - \dim V$ projective hypersurfaces. 
\end{itemize} 

\begin{proof}[Proof (of Lemma \ref{irreducible lemma})] For $0 \leq r \leq m$ let $I_r := \langle f_1, \dots, f_r \rangle$, where it is understood that $I_0 := \{0\}$. It will be shown, using induction, that $I_0 \subset  \dots  \subset I_m$ (with strict inclusion) and that each $I_r$ is a prime ideal. Letting $V_0 := \mathbb{P}^{n-1}(\overline{\mathbb{F}}_p)$ and
\begin{equation*}
V_r := \{x \in \mathbb{P}^{n-1}(\overline{\mathbb{F}}_p) : f_j(x) = 0 \textrm{ for $1 \leq j \leq r$ }\}
\end{equation*}
 for $1 \leq r \leq m$, it then immediately follows that the varieties $V_r$ are all irreducible. It is remarked that it is useful to establish the stronger condition that the $I_r$ are prime in order to facilitate the induction.

The case $r = 0$ (corresponding to the trivial ideal $\{0\}$) is vacuous. Fix $1 \leq r \leq m$ and assume, by way of induction hypothesis, that $I_0 \subset  \dots  \subset I_{r-1}$ and that each $I_i$ is prime for $0 \leq i \leq r-1$. 

To show $I_{r-1} \subset I_{r}$ is a \emph{proper} subset, it suffices to show that $f_r \notin \langle f_1, \dots, f_{r-1} \rangle$. Aiming for a contradiction, suppose that
\begin{equation*}
f_r = \sum_{j=1}^{r-1} h_j f_j
\end{equation*}
 for some $h_j \in \overline{\mathbb{F}}_p[X_1, \dots, X_n]$. Differentiating the above equation,
\begin{equation*}
\frac{\partial f_r}{\partial X_k} = \sum_{j=1}^{r-1} \frac{\partial h_j}{\partial X_k}f_j + h_j \frac{\partial f_j}{\partial X_k}
\end{equation*}
and thus, if $x \in Z(f_1, \dots, f_r)$, then it follows that
\begin{equation*}
\frac{\partial f_r}{\partial X_k}(x) = \sum_{j=1}^{r-1} h_j(x) \frac{\partial f_j}{\partial X_k}(x).
\end{equation*}
However, this identity contradicts the non-degeneracy hypothesis (which implies that the vectors $\big(\frac{\partial f_j}{\partial \vec{X}}(x)\big)_{j=1}^r$ are linearly independent) and so $V_r \subset V_{r-1}$, as claimed.

To prove that $I_r$ is prime it suffices to show:
\begin{enumerate}[i)]
\item $I_r$ is radical;
\item $V_r$ is irreducible.\footnote{Of course, the irreducibility of $V_r$ is the only property that one is really after here, but the stronger condition that $I_r$ is prime is needed to run the induction argument.}
\end{enumerate}
The first step towards proving either of these statements is to show that $\mathrm{height}(I_r) = r$. Recall from the induction hypothesis that $\{0\} = I_0 \subset \dots \subset I_{r-1}$ forms a strictly increasing chain of prime ideals with $I_{r-1} \subset I_r$, which implies that $\mathrm{height}(I_r) \geq r$. Therefore, combining this with the generalised Krull principal ideal theorem, $\mathrm{height}(I_r) = r$, as required.

One may now show that $I_r$ is a radical ideal via a criterion of Serre (see, for instance, \cite[Chapter 18]{Eisenbud1995}). Let $\mathcal{J}_r$ be the ideal generated by the $r \times r$ minors of the Jacobian matrix 
\begin{equation*}
\frac{\partial \vec{f}}{\partial \vec{X}} = \frac{\partial (f_1, \dots, f_r)}{\partial (X_1, \dots, X_n)}
\end{equation*}
taken modulo $I_r$; that is, $\mathcal{J}_r$ is the ideal of the ring $K[X_1, \dots, X_n]/I_r$ generated by the minors of $\partial \vec{f}/\partial \vec{X}$ viewed as elements of  $K[X_1, \dots, X_n]/I_r$. Combining \cite[Proposition 18.13]{Eisenbud1995} and \cite[Proposition 18.15 a)]{Eisenbud1995}, to show $I_r$ is radical it suffices to show that $\mathrm{height}(\mathcal{J}_r) \geq 1$. This is equivalent to showing $\mathrm{height} (I_r+J_r)/I_r \geq 1$ where $J_r$ is the ideal of $K[X_1, \dots, X_n]$ generated by the minors of $\partial \vec{f}/\partial \vec{X}$. The non-degeneracy condition implies that
\begin{equation*}
\mathcal{V}(I_r) \cap \mathcal{V}(J_r) = \big\{ x \in \mathbb{P}^{n-1}(\overline{\mathbb{F}}_p) : f(x) = 0 \textrm{ for all $f \in I_r + J_r$}\big\} = \emptyset 
\end{equation*}
and thus, by the Nullstellensatz, $\sqrt{I_r + J_r} = \langle X_1, \dots, X_n \rangle$ and hence $\mathrm{height}(I_r + J_r) = n$. Since $\langle X_1, \dots, X_n \rangle$ is a maximal ideal of $K[X_1, \dots, X_n]$, it follows from the discussion preceding \eqref{dimension formula} that any maximal chain of prime ideals containing $I_r + J_r$ has length $n$. From this, it follows that
\begin{equation*}
\mathrm{height} (I_r+J_r)/I_r = n -  r \geq 2,
\end{equation*}
and so $I_r$ is radical.

It remains to demonstrate the irreducibility of $V := V_r$; for this it suffices to show the following two conditions hold:
\begin{enumerate}[i)]
\item  $V$ is (Zariski) connected;
\item  $V$ is smooth as a projective variety.
\end{enumerate}
Indeed, any regular point of $V$ lies in precisely 1 irreducible component (see, for instance, \cite[Chapter VI, Proposition 1.13]{Kunz2013}). Thus, if $V$ is smooth, then the irreducible components partition $V$ into disjoint Zariski-closed subsets. If $V$ is also connected, then $V$ must have a single irreducible component, and so $V$ is irreducible. 

By the Nullstellensatz, $\mathrm{height}(\mathcal{I}(V)) = \mathrm{height}(I_r) = r$ and so, recalling \eqref{dimension formula}, it follows that $\dim V = n-1-r\geq 1$. Since $V$ is, by definition, the intersection $r$ of projective hypersurfaces, $V$ is therefore a set-theoretic complete intersection. The Hartshorne connectedness theorem \cite[Chapter VI, Theorem 4.2]{Kunz2013} now implies that $V$ is connected.

Finally, one may verify that $V$ is smooth using the the Jacobian criterion (see, for instance, \cite[Chapter VI, Proposition 1.5]{Kunz2013}\footnote{The reference \cite{Kunz2013} only gives the \emph{affine} Jacobian criterion, rather than the \emph{projective} version used here. Both results, however, can be obtained via a similar method of proof.}). This states that for every point $x \in V$ one has
\begin{equation}\label{Jacobian criterion}
\mathrm{rank} \frac{\partial (f_1, \dots, f_r)}{\partial (X_1, \dots, X_n)} (x) \leq n - 1 - \dim_x V,
\end{equation} 
where $\dim_x V$ is equal to the maximum of the dimensions of the irreducible components of $V$ that contain $x$ and, moreover, if equality holds in \eqref{Jacobian criterion}, then $V$ is smooth at $x$. Note that $n - 1 - \dim V = n - 1 - (n-1-r) = r$, which is precisely the rank of the Jacobian matrix, and so one wishes to show that $\dim_x V = \dim V$. To see this, it suffices to prove for any given $x \in V$ that all the irreducible components of $V$ that contain $x$ have the same dimension. Indeed, in this case, since $V$ is connected, it follows that all the irreducible components of $V$ must have the same dimension, and this must then be equal to $\dim V$ (since $\dim V$ is equal to the maximum of the dimensions of the irreducible components by the Nullstellensatz (see \cite[Chapter II Proposition 3.11]{Kunz2013})). 

Fixing $x \in V$, consider the prime ideal $\mathfrak{p}_x := \{ f \in K[V] : f(x) = 0 \}$ (recall $K[V]$ denotes the co-ordinate ring of $V$). The localisation $K[V]_{\mathfrak{p}_x}$ is Cohen--Macaulay by \cite[Proposition 18.8]{Eisenbud1995} (see also \cite[Chapter VI, Corollary 3.15]{Kunz2013}) and therefore all the minimal primes of $K[V]_{\mathfrak{p}_x}$ have the same dimension $d$ by \cite[Proposition 18.11]{Eisenbud1995}. Let $W$ be an irreducible component of $V$ containing $x$. Thus, by the Nullstellensatz (see \cite[Chapter II Proposition 3.11]{Kunz2013}), $W$ corresponds to a minimal prime ideal $\mathfrak{q} \unlhd  K[V]$ with $\mathfrak{q} \subseteq \mathfrak{p}_x$. The localisation $\mathfrak{q}_{\mathfrak{p}_x}$ is a minimal prime of $K[V]_{\mathfrak{p}_x}$ (see \cite[Chapter III, Proposition 3.6 and Proposition 4.14]{Kunz2013}) and $K[V]_{\mathfrak{p}_x}/\mathfrak{q}_{\mathfrak{p}_x} \cong K[V]/\mathfrak{q}$ (see \cite[Chapter III, Rule 4.15]{Kunz2013}). Combining these observations, $\dim W = \dim K[V]/\mathfrak{q} = d$ and so all irreducible components $W$ containing $x$ have the same dimension, as required.  

\end{proof}

The above argument also yields Corollary \ref{dimension corollary}.

\begin{proof}[Proof (of Corollary \ref{dimension corollary})] The proof of Lemma \ref{irreducible lemma} implies that $\dim V_r = n-1-r$ for $1 \leq r \leq \min\{m,n-2\}$, and it remains only to verify the case $r = m = n-1$. The first part of the argument used to establish the inductive step shows $\mathrm{height}(I_{n-1}) = n-1$ and therefore one deduces that $\dim V_{n-1} = 0$, as required. 
\end{proof}




\bibliography{Reference}
\bibliographystyle{amsplain}

\end{document}